\RequirePackage{amsthm}
\documentclass[sn-mathphys-num]{sn-jnl}

\usepackage{graphicx}%
\usepackage{multirow}%
\usepackage{amsmath,amssymb,amsfonts}%
\usepackage{mathrsfs}%
\usepackage[title]{appendix}%
\usepackage{xcolor}%
\usepackage{textcomp}%
\usepackage{manyfoot}%
\usepackage{booktabs}%
\usepackage{algorithm}%
\usepackage{algorithmicx}%
\usepackage{algpseudocode}%
\usepackage{listings}%
\usepackage{anyfontsize}

\theoremstyle{thmstyleone}%
\newtheorem{theorem}{Theorem}
\newtheorem{proposition}{Proposition}

\theoremstyle{thmstyletwo}%
\newtheorem{remark}{Remark}%

\theoremstyle{thmstylethree}%
\newtheorem{lemma}{Lemma}
\newtheorem{corollary}{Corollary}

\usepackage{comment}

\begin{document}

\title[Limit theorems for a class of random outer measures in infinite urn schemes]{Limit theorems for a class of random outer measures in infinite urn schemes}


\author[1,2]{\fnm{Berhane} \sur{Abebe}}\email{b.andemikael@g.nsu.ru}

\author*[3]{\fnm{Mikhail} \sur{Chebunin}}\email{mikhail.chebunin@uni-ulm.de}

\author[4,1]{\fnm{Artyom} \sur{Kovalevskii}}\email{artyom.kovalevskii@gmail.com}

\affil[1]{ \orgname{Novosibirsk State University}, \orgaddress{\street{Pirogova str., 1}, \city{Novosibirsk}, \postcode{630090}, \state{Novosibirsk Oblast}, \country{Russian Federation}}}

\affil[2]{ \orgname{Mainefhi College of Science}, \orgaddress{\city{Mainefh}, \postcode{}, \state{Zoba Maekel}, \country{Eritrea}}}

\affil[3]{ \orgname{Ulm University}, \orgaddress{\city{Ulm}, \postcode{89081},  \country{Germany}}}

\affil[4]{ \orgname{Sobolev Institute of Mathematics}, \orgaddress{\street{Koptjug ave., 4}, \city{Novosibirsk}, \postcode{630090}, \state{Novosibirsk Oblast}, \country{Russian Federation}}}


\abstract{
An urn scheme is a probabilistic model in which balls are placed into urns sequentially and independently of each other. All balls share the same probability distribution for hitting the urns. In the simplest case, there is a finite number of urns and the probabilities of hitting each urn are equal. In an infinite urn scheme, there is a countable number of urns, and the hitting probabilities form a probability mass function on the set of urn labels, so they depend on the urn number. 

The statistics of interest are the number of urns with at least \(k\ge 1\)  balls after throwing \(n\ge 1\) balls. Thus, we assume that there is a countable family of urns, and we fix the probabilities for a ball to hit each urn (the same for all balls). For an arbitrary subset \(A\) of the unit interval \([0,1]\), we do not consider all ball indices from \(1\) to \(n\), but only those that belong to the set \(nA\), and we study the number of urns with at least \(k\) balls after throwing the balls with indices in \(nA\). This number is non-negative, and if the set \(A\) is empty, this number is equal to zero. 
Moreover, if \(k=1\), then it satisfies the property of countable subadditivity: if \(A\) is contained in a countable union of sets, then the number of non-empty urns after throwing balls with indices in \(nA\) does not exceed the sum of the corresponding numbers of non-empty urns defined in the same way for each of these sets. Hence, the number of non-empty urns for ball indices in \(nA\), where \(A\) is an arbitrary subset of the unit interval, satisfies all the axioms of an outer measure on the unit interval. 

We study the properties of the statistics of interest. Our main result is a functional central limit theorem for sets $A$ consisting of finite unions of intervals and parameterized by their boundary points. We discuss applications of this theorem to elementary probabilistic models of text.}




\maketitle

\section{Introduction}\label{sec1}

Let \(\{X_i\}_{i \geq 1}\) be a family of unbounded, positive integer-valued, independent and identically distributed (i.i.d.) random variables. Consider the \(\sigma\)-algebra generated by the events
\begin{equation}\label{sigma_algebra}
\{X_i = X_j\}_{1 \leq i,j \leq n}, \quad n \geq 1.
\end{equation}

Typical random variables (statistics) that are measurable with respect to this \(\sigma\)-algebra include the number of distinct values \(R_n\) among \(X_1, \dots, X_n\); the number of values occurring exactly \(k\) times \(R_{n,k}\) or at least \(k\) times \(R_{n,k}^*\) for \(k \geq 1\); and similar quantities. We denote by 
\begin{equation}\label{p_i}
p_i := \mathbb{P}(X_1 = i) > 0, \quad i \geq 1.
\end{equation}
Since these statistics are measurable with respect to the \(\sigma\)-algebra generated by the events in \eqref{sigma_algebra}, they depend only on the empirical multiplicity measure and not on the order of the \(X_i\). Thus, we may assume without loss of generality that the probabilities are ordered decreasingly:
\begin{equation}\label{order}
p_i \geq p_{i+1}, \quad i \geq 1.
\end{equation}

Traditionally, this probabilistic model and its statistics are called the infinite urn model or Karlin's occupancy scheme. Samuel Karlin's key role in developing methods for this model is described in detail below. In terms of the infinite urn scheme, the random variable \(X_i\) is the label of the urn hit by the \(i\)-th ball. The number \(R_n\) of distinct values among the first \(n\) random variables is the number of non-empty urns after throwing \(n\) balls; the number \(R^*_{n,k}\) of values occurring at least \(k\) times is the number of urns containing at least \(k\) balls; and so on. 

In the 20th century, there were few papers on infinite urn models. Bahadur~\cite{Bahadur} and Karlin~\cite{Karlin} obtained the basic results. Bahadur was probably the first to study this model: he derived the asymptotic behavior of the expected number of non-empty urns in a particularly important special case and proved a law of large numbers for this quantity. Karlin introduced the regularity assumption 
\begin{equation}\label{reg}
\alpha(x) := \max\left\{i > 0 : p_i \geq 1/x\right\} = x^\theta L(x), \quad \theta \in [0,1],
\end{equation}
where \(L(x)\) is a slowly varying function. Karlin proved the central limit theorem (CLT) for a wide class of statistics under \eqref{reg} with \(\theta > 0\). An important idea of Karlin was Poissonization, i.e., replacing the fixed number \(n\) of random variables by a Poisson number \(\Pi(t)\), \(t \geq 0\), independent of the sequence \(\{X_i\}_{i \geq 1}\). This makes the occupancies of different urns independent Poisson random variables. Karlin also established the strong law of large numbers (SLLN). 

Darling~\cite{Darling} and Key~\cite{Key1992,Key1996} rediscovered the model. Darling obtained results on the asymptotic behavior of the number of non-empty urns and the number of urns with an odd number of balls. Key investigated the asymptotics of the number of singleton urns (those containing exactly one ball). Dutko~\cite{dutko1989central} extended Karlin's research and was the first to prove a central limit theorem for a subclass of distributions satisfying \eqref{reg} with \(\theta = 0\).

In the 21st century, research on Karlin's occupancy scheme has developed intensively in many directions. Reviews of interesting properties of its statistics can be found in \cite{GnedinHansenPitman2007, Ben}. Barbour~\cite{Barbour2009} proposed a translated Poisson approximation for the number of urns containing exactly \(k\) balls (the so-called small counts). The Markov chain properties of these statistics were studied in \cite{muratov}. 

Conditions under which the variance of the number of non-empty urns tends to infinity were investigated in \cite{BarbourGnedin2009}. For a wide class of such distributions, a central limit theorem was proved in \cite{GnedinHansenPitman2007, BarbourGnedin2009, Chang}. Under regularity conditions, functional central limit theorems (FCLT) have been established in \cite{DurieWang2016}, \cite{Chebunin2016}, and \cite{ChebuninZuyev2022}. Durieu 
$\&$ Wang~\cite{Wang} proved a functional central limit theorem for weighted occupancy processes in the Karlin model. 

A law of the iterated logarithm (LIL) for small counts was proved in \cite{iksanov_kotelnikova_2024} under \eqref{reg}, both for \(\theta > 0\) and for \(\theta = 0\) under an additional assumption. Buraczewski et al.~\cite{Buraczewski} established an LIL for the number of occupied urns and related quantities. 

More complex models based on Karlin's occupancy scheme include those with randomly generated probabilities \cite{De_Blasi, Iksanov2022}, and Markov chain-driven infinite urn schemes \cite{Grabchak}. 

In this paper, we consider the restriction of the original \(\sigma\)-algebra \eqref{sigma_algebra} to the \(\sigma\)-algebra generated by the events
\[
\{X_i = X_j\}_{i,j \in \mathbb{Z}_+ \cap nA}, \quad A \subseteq [0,1],
\]
where \(A\) is an arbitrary subset of the unit interval \([0,1]\).

The main object of study is \(R_{nA,k}^*\), the number of distinct urns hit by at least \(k\) balls whose indices lie in \(nA\). We show that for any Karlin model, \(R_{n\cdot,1}^*\) is an integer-valued random outer measure: it vanishes on the empty set \(A = \emptyset\) and satisfies countable subadditivity. For \(k > 1\), \(R_{n\cdot,k}^*\) is no longer subadditive but is majorized by \(R_{n\cdot,1}^*\). 

We then introduce the Poissonization \(R_{tA,k}^{\Pi,*}\) of \(R_{nA,k}^*\), defined as the number of urns containing at least \(k\) balls when the balls are thrown at the points of a standard Poisson process restricted to \(tA\), for \(t > 0\). Here \(A\) is a Borel set, so the Poisson process on \(tA\) is well defined. In this case, \(R_{t\cdot,1}^{\Pi,*}\) is the restriction of a discrete outer measure to the Borel \(\sigma\)-algebra, and the monotonicity \(R_{t\cdot,k+1}^{\Pi,*} \leq R_{t\cdot,k}^{\Pi,*}\) holds for all \(k \geq 1\). 

Lemma~\ref{L3} summarizes these weak properties of \(R_{nA,k}^*\) and its Poissonization. The distribution of \(R_{nA,k}^*\) depends only on the cardinality of \(nA\) and not on the specific locations of its points. Moreover, \(\mathrm{Var}(R_{nA,k}^*) \leq \mathbb{E} R_{nA,k}^*\), a property noted by Bahadur~\cite{Bahadur}, which implies the law of large numbers. The same holds for the Poissonized version. 

Importantly, \(R_{nA,k}^*\) is not monotone in \(n\), even when \(A\) is an interval such as \(A = (a,b]\) with \(0 < a < b\). Moreover, it lacks the difference property and satisfies only the following inequality \(R_{n(a,b],k}^* \geq R_{n[0,b],k}^* - R_{n[0,a],k}^*\). Monotonicity of processes played a key role in proofs of limit theorems starting with Karlin's CLT and SLLN; these methods do not apply to \(R_{nA,k}^*\). 

We prove the central limit theorem for the Poissonized version (Theorem~\ref{CLT}) under the regularity assumption \eqref{reg} for \(\theta \in (0,1)\). This is a ``regular'' multivariate central limit theorem for statistics over a finite collection of arbitrary Borel sets. 

Next, we establish the functional central limit theorem  (Theorem~\ref{FCLT}). This theorem asserts weak convergence, uniform in the interval endpoints, of the standardized processes \(R_{nA,k}^*\) for sets \(A\) consisting of finitely many intervals to a corresponding Gaussian field. The proof relies on estimating the probability of symmetric set differences for the Poissonized version and tightly approximating the original multivariate process by its Poissonization. The corollaries include an analog of Theorem~\ref{CLT} for the non-Poissonized statistics and a result for weighted sum statistics. In particular, Corollary~\ref{Q} generalizes Theorem~1 of \cite{Wang} to a broader class of distributions and multidimensional parameters. 

We also prove the strong law of large numbers under the regularity condition for sets consisting of finitely many intervals. Finally, we discuss applications to elementary probabilistic text models. Modern mathematical linguistics models texts as stochastic processes over an infinite vocabulary rather than points in a finite feature space. Statistics based on the number of distinct words enable diagnostics for concatenation in heterogeneous texts via a test grounded in the  FCLT. 

The paper is organized as follows. We define notations and study basic properties of the statistics in Section~\ref{sec2}. The CLT for Poissonization is proved in Section~\ref{sec3}. The FCLT (Theorem~\ref{FCLT}) is proved in Section~\ref{sec4}. The SLLN is established in Section~\ref{sec5}. Applications to probabilistic text models appear in Section~\ref{sec6}.

\section{Preliminaries}\label{sec2}

Let $A\subseteq [0,\, 1]$ be an arbiratry subset, and $tA=\{ta: \ a \in A\}$ for $t\ge0$. Recall that $\{ X_i\}_{i\ge 1}$ is a family of i.i.d. unbounded random variables on $\mathbb{Z}_+ = \{1,\ 2, \ldots\}$.

Let $R_{nA,k}^*$ be the number of urns that contain at least $k\ge 1$ balls with their sequential numbers in $nA$ for $n\ge 1$,
\begin{equation}
    R_{nA,k}^*:=\sum_{i=1}^{\infty} {\bf 1}\left(\exists \,  m_1^{(i)}<\ldots<m_k^{(i)} \in nA: \   X_{m_1^{(i)}}=\ldots=X_{m_k^{(i)}}=i \right). \nonumber
\end{equation}

We also use special designations $R_n$ and $R^*_{n,k}$ for the case when $A=[0,\, 1]$
\begin{equation}
  R_n := R_{n[0,\,1],1}^*, \ \ R^*_{n,k}:= R_{n[0,\,1],k}^*. \nonumber  
\end{equation}

\begin{lemma}\label{L1}
Let \(n,k \geq 1\) and \(A, A_i\subseteq[0,1]\) for \(i\ge 1\). Then with probability 1
\begin{itemize}
\item[(i)] \(R_{nA,k}^* \geq 0\) and \(R_{n\emptyset,k}^* = 0\).

\item[(ii)] If \(A \subseteq \bigcup_{i=1}^\infty A_i\), then \(R_{nA,1}^* \leq \sum_{i=1}^\infty R_{nA_i,1}^*\).

\item[(iii)] \(R_{nA,k+1}^* \leq R_{nA,k}^* \leq R_{n,k}^* \leq n/k\).
\end{itemize}
\end{lemma}

\begin{proof}
Property (i) follows immediately from the definition. If \(nA \subseteq \bigcup_{i=1}^\infty nA_i\), then every ball index in \(nA\) belongs to some \(nA_i\). Thus, every urn hit at least once by balls from \(nA\) is hit at least once by balls from some \(nA_i\), so the set of such urns for \(nA\) is contained in the union over \(i\) of the sets of urns for each \(nA_i\). Cardinality is countably subadditive, which yields (ii). Property (iii) holds because the event ``urn contains at least \(k+1\) balls from \(nA\)'' is contained in ``at least \(k\) balls from \(nA\)''; combined with the previous inequality and a simple bounding \(R_{n,k}^*\) (at most \(n\) balls into urns with $\geq k$ each). 
\end{proof}

\begin{remark}\label{R1}
By Lemma~\ref{L1}(i)--(ii), for each fixed \(n \geq 1\), the map \(A \mapsto R_{nA,1}^*\) is a random integer-valued outer measure on \([0,1]\). This fails for \(R_{n\cdot,k}^*\) when \(k > 1\), as (ii) need not hold. Indeed, for \(n=2\) and \(A = [0,1] = A_1 \cup A_2\) with \(A_1 = [0,1/2]\), \(A_2 = (1/2,1]\), if both balls 1 and 2 fall into the same urn, then \(R_{nA,2}^* = 1\) while \(R_{nA_1,2}^* = R_{nA_2,2}^* = 0\). 
\end{remark}

Let \(\{\Pi(t), t \geq 0\}\) be a Poisson process with intensity 1, independent of the family \(\{X_i\}_{i \geq 1}\). Let \(0 < T_1 < T_2 < \ldots\) denote its points (arrival times). We define the Poissonized analog of \(R_{nA,k}^*\) by
\begin{equation}
R_{tA,k}^{\Pi,*} := \sum_{i=1}^\infty \mathbf{1}\Bigl( \exists \, T_{m_1^{(i)}} < \cdots < T_{m_k^{(i)}} \in tA : X_{m_1^{(i)}} = \cdots = X_{m_k^{(i)}} = i \Bigr)\nonumber
\end{equation}
for \(t > 0\). When \(A = [0,1]\), we write \(R_{\Pi(t)} := R_{t[0,1],1}^{\Pi,*}\) and \(R_{t,k}^{\Pi,*} := R_{t[0,1],k}^{\Pi,*}\).

The Poisson process is well-defined on \(tA\) provided \(A\) is a Borel set. Let \(\mathcal{B}[0,1]\) denote the Borel \(\sigma\)-algebra on \([0,1]\). 

\begin{lemma}\label{L2}
Let \(t > 0\), \(k \geq 1\), and \(A, A_i \in \mathcal{B}[0,1]\) for $i\ge 1$. Then with probability 1
\begin{itemize}
\item[(i)] \(R_{tA,k}^{\Pi,*} \geq 0\) and \(R_{t\emptyset,k}^{\Pi,*} = 0\).

\item[(ii)] If \(A \subseteq \bigcup_{i=1}^\infty A_i\), then \(R_{tA,1}^{\Pi,*} \leq \sum_{i=1}^\infty R_{tA_i,1}^{\Pi,*}\).

\item[(iii)] \(R_{tA,k+1}^{\Pi,*} \leq R_{tA,k}^{\Pi,*} \leq R_{t,k}^{\Pi,*} \leq \Pi(t)/k\).
\end{itemize}
\end{lemma}

\begin{proof}
The result follows from Lemma~\ref{L1} by replacing the first \(n\) sequential ball indices \(\{1,\dots,n\}\) with the points \(\{T_j : T_j \leq t\}\) of the Poisson process up to time \(t\).
\end{proof}

\begin{remark}\label{R2}
Analogously to Remark~\ref{R1}, for each fixed \(t > 0\), the map \(A \mapsto R_{tA,1}^{\Pi,*}\) is a random integer-valued outer measure on \([0,1]\), but this fails for \(k > 1\). 
\end{remark}

Let \(\#(nA)\) denote the cardinality of the set \(nA \cap \mathbb{Z}_+\). We write \(|A|\) for the Lebesgue measure of a Borel set \(A \subseteq [0,1]\) and \(\stackrel{d}{=}\) for equality in distribution.

Lemma~\ref{L3} collects weak and moment properties of the statistics, including convergence of expectations to infinity and laws of large numbers.

\begin{lemma}\label{L3}
Let \(A \subset[0,1]\), \(n,k \geq 1\). Then
\begin{itemize}
\item[(A1)] \(R_{nA,k}^* \stackrel{d}{=} R_{\#(nA),k}^*\).

\item[(A2)] \(\mathbb{E} R_{nA,k}^* = \mathbb{E} R_{\#(nA),k}^*\) and \(\mathrm{Var}(R_{nA,k}^*) \leq \mathbb{E} R_{nA,k}^*\).

\item[(A3)] If \(\#(nA) \to \infty\) as \(n \to \infty\), then \(\mathbb{E} R_{nA,k}^* \to \infty\) and \(R_{nA,k}^*/\mathbb{E} R_{nA,k}^* \to 1\) in probability.
\end{itemize}

Let \(A \in \mathcal{B}[0,1]\), \(t > 0\), \(k \geq 1\). Then
\begin{itemize}
\item[(B1)] \(R_{tA,k}^{\Pi,*} \stackrel{d}{=} R_{t|A|,k}^{\Pi,*}\).

\item[(B2)] \(\mathbb{E} R_{tA,k}^{\Pi,*} = \mathbb{E} R_{t|A|,k}^{\Pi,*}\) and \(\mathrm{Var}(R_{tA,k}^{\Pi,*}) \leq \mathbb{E} R_{tA,k}^{\Pi,*}\).

\item[(B3)] If \(|A| > 0\), then as \(t \to \infty\), \(\mathbb{E} R_{tA,k}^{\Pi,*} \to \infty\) and \(R_{tA,k}^{\Pi,*}/\mathbb{E} R_{tA,k}^{\Pi,*} \to 1\) in probability.
\end{itemize}
\end{lemma}

\begin{proof}
Property (A1) follows from the definitions of \(R_{nA,k}^*\) and \(\#(nA)\). Thus, (A2) and (A3) are known results of Bahadur~\cite{Bahadur} and Karlin~\cite{Karlin}. For (B1), the Poisson process on \(tA\) has intensity \(t|A|\)  by the splitting property of Poisson processes. Hence, (B2) and (B3) follow from (B1) and the results of Karlin~\cite{Karlin}. 
\end{proof}

\begin{remark}\label{R3}
Note that \(R_{nA,1}^*\) is a configuration function associated with the ``distinctness'' property; see \cite[Section 3.3]{Boucheron}. Thus, \cite[Corollary 3.8]{Boucheron} yields \(\mathrm{Var}(R_{nA,1}^*) \leq \mathbb{E} R_{nA,1}^*\). However, for \(k > 1\), \(R_{nA,k}^*\) is not a configuration function. Moreover, it is not self-bounded in the sense of \cite[Corollary 3.7]{Boucheron}. 
\end{remark}

\begin{remark}\label{R4}
The assumptions of Lemma~\ref{L3}(B3) do not imply \(\#(nA) \to \infty\). A key example is \(A = [0,1] \setminus \mathbb{Q}\): then \(|A| = 1\) but \(A\) contains no rationals, so \(\#(nA) = 0\) for all \(n \in \mathbb{Z}_+\).

For Jordan measurable \(A\), we have \(\#(nA)/n \to |A|\). Indeed, \(A\) is Jordan measurable if and only if its topological boundary has Lebesgue measure zero, which (by the Lebesgue--Vitali covering theorem and properties of Riemann sums) is equivalent to the Riemann integrability of \(\mathbf{1}_A\), with \(\#(nA)/n\) as a Riemann sum approximation. 
\end{remark}

\section{Central Limit Theorem for the Poissonization}\label{sec3}

In this section, we prove CLT for the Poissonized statistics. Theorem~\ref{CLT} significantly generalizes Karlin's \cite[Theorem 4]{Karlin} CLT for Poissonized occupancies: rather than the full interval \([0,t]\), we consider arbitrary Borel sets \(tA\). Moreover, Theorem~\ref{CLT} provides the explicit covariances of the limiting Gaussian vector, which coincide with those in the functional central limit theorem (Theorem~\ref{FCLT}) of the next section. 
Denote by 
\begin{align*}
R_{nA,k}&:=R_{nA,k}^* - R_{nA,k+1}^*, \ \   
R^{\Pi}_{tA,k}:=R^{\Pi,*}_{tA,k}-R^{\Pi,*}_{tA,k+1},\\
    R_{n,k}&:=R_{n[0,1],k}, \ \   
R^{\Pi}_{t,k}:=R^{\Pi}_{t[0,1],k},\\
M(t)&:=\mathbb{E} R_{\Pi(t)}, \ \ M_k(t):=\mathbb{E} R^{\Pi}_{t,k}.
\end{align*}

By Lemma \ref{L3} (B2), if $A\in \mathcal{B}[0,\, 1]$ and $|A|>0$, then $\mathbb{E}\, R^{\Pi}_{tA,k}=M_k(t|A|)$.
Let $\Pi_i(A)$ be the number of balls in the Poissonized
version of the process in urn $i$ with times from 
Borel set $A$. Hence $\Pi_i(A)$ has a Poisson distribution with parameter $p_i|A|$.

Define the covariances of Poissonized statistics by
\[{c}_{k_1, k_2} (A_1,A_2):={\bf cov} (R^{\Pi}_{A_1,k_1}, R^{\Pi}_{A_2,k_2}), \ \  {c}_{k_1, k_2}^* (A_1,A_2):={\bf cov} (R^{\Pi,*}_{A_1,k_1}, R^{\Pi,*}_{A_2,k_2}).\]

\begin{proposition}\label{corr}
Let $A_1,A_2  \in \mathcal{B}[0,\, \infty)$,
 and $k_1,k_2\ge 1$.
Then
\begin{align*}
    {c}_{k_1, k_2} (A_1,A_2) 
=\sum_{i=1}^{\infty} {\bf cov}({\bf 1}(\Pi_i(A_1)=k_1), {\bf 1}(\Pi_i (A_2)=k_2)),
\end{align*}
and
\begin{align*}{c}^*_{k_1, k_2} (A_1,A_2) 
=&\sum_{i=1}^{\infty} {\bf cov}({\bf 1}(\Pi_i(A_1)\ge k_1), {\bf 1}(\Pi_i (A_2)\ge k_2))\\ 
=&\sum_{i=1}^{\infty} {\bf cov}({\bf 1}(\Pi_i(A_1)< k_1), {\bf 1}(\Pi_i (A_2)< k_2)).\end{align*}
Moreover, if $|A_1|<\infty$ and $|A_2|<\infty$, then
 $c^*_{1,1} (A_1,A_2)=M(|A_1|+|A_2|)-M(|A_1 \cup A_2|)$.

\end{proposition}

\begin{proof}
The statement follows from the definitions and simple observations.
\end{proof}

Karlin proved that if the regularity assumption (\ref{reg}) holds for $\theta\in(0,1)$ then for any $k\ge 1$ as $t \to \infty$
\begin{equation}\label{conv}
M(t)\sim\Gamma(1-\theta)t^{\theta}L(t), \ \ \ M_k(t) \sim \frac{\theta \Gamma(k-\theta)}{k!} t^{\theta}L(t).
\end{equation}
Therefore, as $t \to \infty$
\begin{equation}\label{Mkto}
\frac{M_k(t)}{M(t)} \to q_k:=(-1)^{k+1}\binom{\theta}{k},
\end{equation}
 where $q_k$ is the Karlin-Rouault probability mass function, see \cite{Karlin,Janson,Naulet}.

Denote the weak convergence by $\Rightarrow$, and define the following processes by
\begin{equation}\label{Y}
Y_{t,A,k}^{\Pi,*} := \frac{R_{tA,k}^{\Pi,*} - \mathbb{E} \, R_{tA,k}^{\Pi,*}}{\sqrt{M(t)}},
\ \ \ \
Y_{n,A,k}^* := \frac{R_{nA,k}^* - \mathbb{E} \, R_{nA,k}^*}{\sqrt{M(n)}}. 
\end{equation}

\bigskip

\begin{theorem}[CLT for Poissonization]\label{CLT}
Let \(k_1, \dots, k_m \geq 1\) and \(A_1, \dots, A_m \in \mathcal{B}[0,1]\) for some \(m \geq 1\), and assume the regularity condition \eqref{reg} holds for some \(\theta \in (0,1)\). Then, as \(t \to \infty\),
\[
\Bigl( Y_{t,A_1,k_1}^{\Pi,*}, \dots, Y_{t,A_m,k_m}^{\Pi,*} \Bigr) \Rightarrow \Bigl( Y_{A_1,k_1}^*, \dots, Y_{A_m,k_m}^* \Bigr),
\]
where the limit is a centered Gaussian vector with covariances
\begin{equation}\label{K*k1}
K^*(A_1,k_1,A_2,k_2) := \frac{1}{\Gamma(1-\theta)} \int_0^\infty \mathrm{cov} \bigl( \mathbf{1}(\Pi(uA_1) < k_1), \mathbf{1}(\Pi(uA_2) < k_2) \bigr) \, d(-u^{-\theta}).
\end{equation}
\end{theorem}

\begin{proof}

Denote \(A_1 \setminus A_2\), \(A_1 \cap A_2\), and \(\overline{A_1} \cap A_2\) by \(B_1\), \(B_2\), and \(B_3\), respectively. Note that \(A_1 \cup A_2 = B_1 + B_2 + B_3\) (disjoint union). By definition \eqref{Y} and Proposition~\ref{corr}, we have
\begin{align*}
K^*(A_1,k_1,A_2,k_2)&= \lim_{t \to \infty} \frac{c^*_{k_1,k_2} (tA_1,tA_2)}{M(t)}\\ 
&=\lim_{t \to \infty} \frac{1}{M(t)}
\sum_{i=1}^{\infty} {\bf cov}\bigg({\bf 1}(\Pi_i(tA_1)< k_1), {\bf 1}(\Pi_i (tA_2)< k_2)\bigg). 
\end{align*}
For fixed $i\ge 1$ we have
\begin{align*}
    &{\bf cov}\bigg({\bf 1}(\Pi_i(tA_1)< k_1), {\bf 1}(\Pi_i (tA_2)< k_2)\bigg)\\
&= 
\sum_{m=0}^{\min(k_1,k_2)} 
\mathbb{P} (\Pi_i(tB_1)<k_1-m, \ \Pi_i(tB_2)=m, \ \Pi_i(tB_3)<k_2-m) \\
& \qquad\qquad\qquad\qquad\qquad\qquad\qquad\qquad\qquad- \mathbb{P} (\Pi_i(tA_1)<k_1) \ \mathbb{P}(\Pi_i(tA_2)<k_2)\\
&=\sum_{k=0}^{k_1+k_2-2}  \left( \nu_{k,1} (tp_i)^k e^{-tp_i|A_1\cup A_2|}
-\nu_{k,2} (tp_i)^k e^{-tp_i(|A_1|+|A_2|)}
\right),
\end{align*}
where corresponding coefficients $\nu_{k,1}$, $\nu_{k,2}$ depend on $k, \ |B_1|, \ |B_2|$, and $|B_3|$. Note that $M(t)=\sum_{i=1}^{\infty}(1-e^{-tp_i})$ and $\nu_{0,1}=\nu_{0,2}=1$. Hence,
\begin{align}\notag
    &K^*(A_1,k_1,A_2,k_2)\\\notag &= \lim_{t \to \infty} \frac{1}{M(t)}\bigg( c^*_{1,1}(tA_1, tA_2) +\sum_{k=1}^{k_1+k_2-2}  \left( \mu_{k,1}  M_k (t|A_1\cup A_2|)
-\mu_{k,2} M_k (t(|A_1|+|A_2|))\right)\bigg)\\ \label{K_fin}
&=(|A_1|+|A_2|)^{\theta}-|A_1\cup A_2|^{\theta} 
+\sum_{k=1}^{k_1+k_2-2}  \left( \mu_{k,1}  |A_1\cup A_2|^{\theta} 
-\mu_{k,2} (|A_1|+|A_2|)^{\theta}
\right) q_k,
\end{align}
where $\mu_{k,1}=\nu_{k,1} k! |A_1\cup A_2|^{-k}$ and $\mu_{k,2}=\nu_{k,2} k! (|A_1|+|A_2|)^{-k}$ for $k\ge 1$.
By simple observation
\[
q_k=\frac{1}{\Gamma(1-\theta)}
 \int_0^{\infty} \mathbb{P} (\Pi(u)=k)\, d(-u^{-\theta}),
\]
with substitutions $u=t|A_1\cup A_2|$ and
$u=t(|A_1|+|A_2|)$, we get (\ref{K*k1}) from (\ref{K_fin}).

Weak convergence of the Poissonization $\left(
Y_{t,A_1,k_1}^{\Pi,*}, \ldots, Y_{t,A_m,k_m}^{\Pi,*}
\right) $ to the limiting random variable follows from the modification of the Dutko \cite{dutko1989central} argument: vectors of indicators
\[
\bigg({\bf 1}(\text{there are not lesser than } k_r 
\text{ balls } j \text{ in urn } i 
\text{ with } t_j \in tA_r)\bigg)_{r=1,\ldots,m}
\]
are independent for different $i$ (due to the splitting property of the Poisson process) and satisfy the Lindeberg condition.
\end{proof}

Taking $k_1=k_2=1$, from (\ref{K_fin}) we get the following corollary.
\begin{corollary}\label{K*}
 $
 K^*(A_1,1,A_2,1)= (|A_1|+|A_2|)^{\theta}-|A_1 \cup A_2|^{\theta}
$.
\end{corollary}
\begin{corollary}\label{KAA}
Let $A \in \mathcal{B}[0,\, 1]$, $i,j\ge 1$. Then
\[
K(A,i,A,j):= \lim_{t \to \infty} \frac{c_{ij} (tA_1,tA_2)}{M(t)} = 
\pi_{ij}|A|^{\theta},
\]
where
\[
\pi_{ij}:=
{\bf 1}(i=j) 
q_{i}
\ -  \binom{i+j}{i} 2^{\theta-i-j} q_{i+j}.
\]
\end{corollary}
\begin{proof}
By Theorem \ref{CLT} we have
\begin{align*}
    K(A,i,A,j)&=
\frac{1}{\Gamma(1-\theta)}
 \int_0^{\infty} {\bf cov} \big( {\bf 1}(\Pi(tA)=i), {\bf 1}(\Pi(tA)= j)\big)
d(-t^{-\theta})=\pi_{ij}|A|^{\theta}.
\end{align*}
\end{proof}

\section{Functional Central Limit Theorem}\label{sec4}

In this section, we prove FCLT. Theorem~\ref{FCLT} extends \cite[Theorem 1]{Chebunin2016} to multidimensional time. The main difficulty is that the statistics \(R_{nA,k}^*\) are not monotone in \(n\), so classical arguments based on monotonicity are not available. We begin by analyzing the properties of Poissonized statistics. 

For \(k \geq 1\) and \(A,B \in \mathcal{B}[0,\infty)\) with \(|A| < \infty\) and \(|B| < \infty\), define
\[
R_{B,k}^{\Pi,*} - R_{A,k}^{\Pi,*}
= \sum_{i=1}^{\infty} \Big( \mathbf{1}\big(\Pi_i(B) \geq k\big) - \mathbf{1}\big(\Pi_i(A) \geq k\big) \Big)
=\sum_{i=1}^{\infty} \mathbf{1}_{i,k}(A,B),
\]
where \(\mathbf{1}_{i,k}(A,B):=\mathbf{1}\big(\Pi_i(B) \geq k\big) - \mathbf{1}\big(\Pi_i(A) \geq k\big)\), and set \(E_{i,k}(A,B) := \mathbb{E}\,\mathbf{1}_{i,k}(A,B)\). For sets \(A,B\subset \mathbb{R}\), we denote their symmetric difference by 
\[
A \Delta B := (A \setminus B) \cup (B \setminus A).
\]

\begin{lemma}\label{L_F_1}
    Let  \(i, k \geq 1\) and \(A,B \in \mathcal{B}[0,\infty)\). Suppose that \(|A| < \infty\) and \(|B| < \infty\), then
    \begin{equation}\label{E_1}
        |{\bf 1}_{i,k}(A,B)|\le_\text{a.s.} {\bf 1}(\Pi_i(A \Delta  B)\ge 1),
    \end{equation}
    
    \begin{equation}\label{E_2}
        |E_{i,k}(A,B)|\le {\bf P}(\Pi_i(A\Delta B)\ge 1)\le p_i|A \Delta B|,
    \end{equation}
   and for any $m\ge 1$
    \begin{equation}\label{E_3}
       \mathbb{E} |{\bf 1}_{i,k}(A,B) - E_{i,k}(A,B)|^m\le (2^m+1){\bf P}(\Pi_i(A\Delta B)\ge 1).
    \end{equation}
    
\end{lemma}

\begin{proof}
    Note that 
    \begin{align*}
        {\bf 1}_{i,k}(A,B)=\sum_{j=0}^{k-1} {\bf 1}(\Pi_i(AB)=j) \bigg[{\bf 1}&(\Pi_i(B\overline{A})\ge k-j), \Pi_i(A\overline{B})< k-j)) \\
    - {\bf 1}&(\Pi_i(B\overline{A})< k-j), 
    \Pi_i(A\overline{B})\ge k-j)\bigg]. 
    \end{align*}
    Therefore, with probability one
    \begin{align}\label{E_I}
        |{\bf 1}_{i,k}(A,B)|\le 
     \max \bigg({\bf 1}(\Pi_i(B\overline{A})\ge 1),{\bf 1}(\Pi_i(A\overline{B})\ge 1)\bigg)=  
     {\bf 1}(\Pi_i(A \Delta  B)\ge 1).  
    \end{align}
     Inequality \eqref{E_2} follows from \eqref{E_I}, since
    \[
    {\bf P}(\Pi_i(A\Delta B)\ge 1)=1-e^{-p_i |A \Delta B|} \le p_i|A \Delta B|.
    \]
    Moreover, by \eqref{E_I} and \eqref{E_2} we have
    \begin{align*}
       \mathbb{E} |{\bf 1}_{i,k}(A,B) -  E_{i,k}(A,B)|^m 
       &\le (1 + E_{i,k}(A,B))^m {\bf P}(\Pi_i(A\Delta B)\ge 1) +(E_{i,k}(A,B))^m \\
       &\le (2^m+1){\bf P}(\Pi_i(A\Delta B)\ge 1).
    \end{align*} 
    
\end{proof}

\begin{corollary}\label{C_F_1}
    We have 
    \begin{align*}
        |R_{B,k}^{\Pi,*} - R_{A,k}^{\Pi,*}|\le_\text{a.s.} R_{A \Delta  B,1}^{\Pi,*}.
    \end{align*}
\end{corollary}
\begin{proof}
    The statement follows from \eqref{E_1}.
\end{proof}

For \(d \geq 1\), we parameterize finite unions of \(d\) intervals in \([0,1]\) by
\[
\mathbf{t} = (\underline{t}_1, \dots, \underline{t}_d, \overline{t}_1, \dots, \overline{t}_d) \in \mathcal{I}_d \subset [0,1]^{2d},
\]
where \(\mathcal{I}_d := \{\underline{t}_i \leq \overline{t}_i \text{ for all } 1 \leq i \leq d\}\). For \(\mathbf{x} \in \mathbb{R}^{2d}\), define
\[
\|\mathbf{x}\| := \|\mathbf{x}\|_d := \max_{1 \leq i \leq 2d} |x_i|.
\]
For \(a,b \geq 0\), let \(I(a,b) := [a,b]\mathbf{1}(a \leq b) + [b,a]\mathbf{1}(a > b)\). For \(\mathbf{s}, \mathbf{t} \in \mathcal{I}_d\), define the associated sets
\[
A_{\mathbf{t}} := \bigcup_{i=1}^d [\underline{t}_i, \overline{t}_i], \qquad \mathbf{s} \Delta \mathbf{t} := A_{\mathbf{s}} \Delta A_{\mathbf{t}}.
\]

\begin{lemma}\label{L_F_2}
Let \(\mathbf{s}, \mathbf{t} \in \mathcal{I}_d\) for \(d \geq 1\). Then
\[
\mathbf{s} \Delta \mathbf{t} \subseteq \bigcup_{i=1}^d \bigl( I(\underline{s}_i, \underline{t}_i) \cup I(\overline{s}_i, \overline{t}_i) \bigr),
\]
and \(|\mathbf{s} \Delta \mathbf{t}| \leq 2d \|\mathbf{s} - \mathbf{t}\|\).
\end{lemma}

\begin{proof}
First, suppose \(d=1\). If \(A_{\mathbf{s}} \cap A_{\mathbf{t}} \neq \emptyset\), then
\[
|A_{\mathbf{s}} \Delta A_{\mathbf{t}}| = |\underline{s}_1 - \underline{t}_1| + |\overline{s}_1 - \overline{t}_1| \leq 2 \|\mathbf{s} - \mathbf{t}\|.
\]
Otherwise, \(A_{\mathbf{s}} \cap A_{\mathbf{t}} = \emptyset\), so
\[
|A_{\mathbf{s}} \Delta A_{\mathbf{t}}| = |\underline{s}_1 - \overline{s}_1| + |\underline{t}_1 - \overline{t}_1| < 2 \|\mathbf{s} - \mathbf{t}\|.
\]
The general case follows since the symmetric difference of unions is contained in the union of symmetric differences.
\end{proof}

\begin{corollary}\label{C_F_2}
Let \(\mathbf{s}, \mathbf{t} \in \mathcal{I}_d\) for \(d \geq 1\) and \(n \geq 1\). Then
\[
|R_{n A_{\mathbf{s}},k}^{\Pi,*} - R_{n A_{\mathbf{t}},k}^{\Pi,*}| \leq_\text{a.s.} R_{n(\mathbf{s} \Delta \mathbf{t}),1}^{\Pi,*} \leq_d R_{2nd \|\mathbf{s} - \mathbf{t}\|,1}^{\Pi,*}.
\]
\end{corollary}

\begin{proof}
This follows immediately from Corollary~\ref{C_F_1}, Lemma~\ref{L_F_2}, and Lemma~\ref{L3}(B1).
\end{proof}

\begin{lemma}\label{L_F_3}
Let \(L > 0\). For \(n \geq 1\), define
\[
\xi_n(L) := \sup_{0 \leq t \leq 1 - L/n} \bigl( \Pi(nt + L) - \Pi(nt) \bigr).
\]
Then \(\xi_n(L)/\sqrt{M(n)} \to 0\) almost surely as \(n \to \infty\).
\end{lemma}

\begin{proof}
Partition \([0,n]\) into \(\lceil n/L \rceil\) subintervals of length \(L\) (except possibly the last). Let \(\eta_k := \Pi(kL) - \Pi((k-1)L)\) for \(k \geq 1\). Then the \(\{\eta_k\}_{k \geq 1}\) are i.i.d. Poisson distributed random variables with parameter \(L\). Clearly, \(\xi_n(L) \leq_\text{a.s.} 2 \max\{\eta_1, \dots, \eta_{\lceil n/L \rceil}\}\).  The claim follows from the tail asymptotics of the Poisson maximum. 
\end{proof}

Recall that \(n\mathbf{t} = (n\underline{t}_1, \dots, n\underline{t}_d, n\overline{t}_1, \dots, n\overline{t}_d)\) for \(\mathbf{t} \in \mathcal{I}_d, \ n\ge 1\). For \(\delta > 0\), let \(\mathcal{I}_{d,\delta} \subset [0,1+\delta]^{2d}\) be the set of points \(\mathbf{s}\), such that \(\underline{s}_i \leq \overline{s}_i\) for all \(1 \leq i \leq d\). Define
\begin{align*}
\Pi(n\mathbf{s}) &:= \bigl( \Pi(n\underline{s}_1), \dots, \Pi(n\underline{s}_d), \Pi(n\overline{s}_1), \dots, \Pi(n\overline{s}_d) \bigr), \\
[n\mathbf{t}] &:= \bigl( [n\underline{t}_1], \dots, [n\underline{t}_d], [n\overline{t}_1], \dots, [n\overline{t}_d] \bigr).
\end{align*}

\begin{lemma}\label{L_F_4}
For any \(\varepsilon, \delta \in (0,1)\), there exists \(n_0 = n_0(\varepsilon, \delta)\) such that for all \(n \geq n_0\),
\[
\mathbf{P}\Bigl( \forall \mathbf{t} \in \mathcal{I}_d \, \exists \mathbf{s} \in \mathcal{I}_{d,\delta} : \|\mathbf{s} - \mathbf{t}\| \leq \delta,\, \Pi(n\mathbf{s}) = [n\mathbf{t}] \Bigr) =: \mathbf{P}(A(n)) \geq 1 - \varepsilon.
\]
\end{lemma}

\begin{proof}
By \cite[Lemma 1(iii)]{Chebunin2016}, for any \(d \geq 1\) and \(\varepsilon, \delta \in (0,1)\), there exists \(n_0'\) such that for \(n \geq n_0'\),
\[
\mathbf{P}\Bigl( \forall t \in [0,1] \, \exists \tau \in [0,1+\delta] : |\tau - t| \leq \delta/(2d), \, \Pi(n\tau) = [nt] \Bigr) \geq 1 - \varepsilon/(2d).
\]
Thus, there exists \(n_0\) such that for \(n \geq n_0\) and any \(\underline{t}_i \leq \overline{t}_i\) (\(1 \leq i \leq d\)), there exist \(\underline{s}_i, \overline{s}_i\) satisfying
\[
|\underline{s}_i - \underline{t}_i| \leq \frac{\delta}{2d}, \quad |\overline{s}_i - \overline{t}_i| \leq \frac{\delta}{2d}, \quad \Pi(n\underline{s}_i) = [n\underline{t}_i], \quad \Pi(n\overline{s}_i) = [n\overline{t}_i],
\]
each with probability at least \(1 - \varepsilon/d\). Monotonicity of \(\Pi\) ensures \(\underline{s}_i \leq \overline{s}_i\), so \(\mathbf{s} \in \mathcal{I}_{d,\delta}\) and \(\|\mathbf{s} - \mathbf{t}\| \leq \delta\). The result follows by union bound over \(i=1,\dots,d\).
\end{proof}

For \(\mathbf{t} \in \mathcal{I}_d\), \(n \geq 1\), and \(k \geq 1\), define the standardized processes
\begin{equation}\label{Y_t_P}
Y_{n,k}^{\Pi,*}(\mathbf{t}) := \frac{R_{n A_{\mathbf{t}},k}^{\Pi,*} - \mathbb{E} R_{n A_{\mathbf{t}},k}^{\Pi,*}}{\sqrt{M(n)}}, \qquad
Y_{n,k}^{*}(\mathbf{t}) := \frac{R_{n A_{\mathbf{t}},k}^{*} - \mathbb{E} R_{n A_{\mathbf{t}},k}^{*}}{\sqrt{M(n)}}.
\end{equation}
For a function \(f: \mathcal{I}_d \to \mathbb{R}\), let \(\omega_f(\delta) := \sup_{\|\mathbf{s} - \mathbf{t}\| \leq \delta} |f(\mathbf{s}) - f(\mathbf{t})|\) be its modulus of continuity. 

\begin{theorem}[Multidimensional FCLT]\label{FCLT}
Let \(d, m \geq 1\), \(k_1, \dots, k_m \geq 1\), and assume the regularity condition \eqref{reg} holds for some \(\theta \in (0,1)\). Then, as \(n \to \infty\),
\begin{equation}
\bigl( Y_{n,k_1}^{\Pi,*}(\mathbf{t}_1), \dots, Y_{n,k_m}^{\Pi,*}(\mathbf{t}_m) : \mathbf{t}_j \in \mathcal{I}_d,\, j=1,\dots,m \bigr)
\Rightarrow \bigl( Y_{k_1}^*(\mathbf{t}_1), \dots, Y_{k_m}^*(\mathbf{t}_m) : \mathbf{t}_j \in \mathcal{I}_d \bigr),\nonumber
\end{equation}
and
\begin{equation}
\bigl( Y_{n,k_1}^{*}(\mathbf{t}_1), \dots, Y_{n,k_m}^{*}(\mathbf{t}_m) : \mathbf{t}_j \in \mathcal{I}_d,\, j=1,\dots,m \bigr)
\Rightarrow \bigl( Y_{k_1}^*(\mathbf{t}_1), \dots, Y_{k_m}^*(\mathbf{t}_m) : \mathbf{t}_j \in \mathcal{I}_d \bigr).\nonumber
\end{equation}
Here, the limit is an \(m\)-dimensional Gaussian field on \(\mathcal{I}_d\) with mean zero and covariance function given by \eqref{K*k1} in Theorem~\ref{CLT}. Moreover, the limiting field has continuous paths almost surely. 
\end{theorem}

\begin{proof}

To establish weak convergence to a continuous Gaussian field, we verify the conditions of \cite[Theorem 1]{Davydov2008}. Theorem~\ref{CLT} provides weak convergence of the finite-dimensional distributions for the Poissonized processes \(Y_{n,k}^{\Pi,*}(\mathbf{t})\). Fix any \(k \in \{k_1, \dots, k_m\}\). We first check Davydov's conditions for the Poissonized field \(Y_{n,k}^{\Pi,*}\), then show that \(Y_{n,k}^{\Pi,*}(\mathbf{t})\) and \(Y_{n,k}^{*}(\mathbf{t})\) are close in an appropriate sense. By Corollary~\ref{C_F_2} and Lemma~\ref{L_F_2}, we have 
\begin{align*}
    \omega_{Y_{n,k}^{\Pi,*}}(1/n)\le& \sup\limits_{\| {\bf s} -{\bf t}\|\le 1/n} \left( |R_{ n A_{\bf s},k}^{\Pi,*}-R_{ n A_{\bf t},k}^{\Pi,*}|+|\mathbb{E}( R_{ n A_{\bf s},k}^{\Pi,*}-R_{ n A_{\bf t},k}^{\Pi,*})|\right)/\sqrt{M(n)} \\
    \le & \sup\limits_{\| {\bf s} -{\bf t}\|\le 1/n} \left( R_{n( {\bf s} \Delta{\bf t}),1}^{\Pi,*}+ 4d\right)/\sqrt{M(n)} \le  2 d (\xi_n(2)+ 2)/\sqrt{M(n)}.
\end{align*}
 Therefore, $\omega_{Y_{n,k}^{\Pi,*}}(1/n)\to 0$ a.s. as $n\to\infty$ by Lemma \ref{L_F_3}.
  
Let ${\bf s}, {\bf t} \in {\mathcal I}_d$ such that $\| {\bf s} -{\bf t}\|\ge 1/n$ and $\gamma=\lceil 8(d+1)/\theta\rceil$. Denote by ${\bf 1}_{i,k}:= {\bf 1}_{i,k}(nA_{\bf s},nA_{\bf t})$ and $E_{i,k}:=\mathbb{E} {\bf 1}_{i,k}$.
Using independence of terms and  Rosenthal inequality we have
\begin{align*}
    \mathbb{E}|Y_{n, k}^{*}({\bf t})-Y_{n, k}^{*}({\bf s})|^{\gamma}\le& \frac{c(\gamma)}{(M(n))^{\gamma/2}}
\left(
\sum\limits_{i=1}^{\infty} \mathbb{E} | {\bf 1}_{i,k}-E_{i,k}|^{\gamma}+
\left( \sum\limits_{i=1}^{\infty} \mathbb{E} ( {\bf 1}_{i,k}-E_{i,k})^2\right)^{\gamma/2}\right) \\
\le&
 \frac{c(\gamma)3^\gamma}{(M(n))^{\gamma/2}}
\left( \mathbb{E} R_{n( {\bf s} \Delta{\bf t}),1}^{\Pi,*} + \left(
 \mathbb{E} R_{n( {\bf s} \Delta{\bf t}),1}^{\Pi,*} \right)^{\gamma/2}\right) \\
\le&
\frac{c(\gamma)3^\gamma}{(M(n))^{\gamma/2}}\left(4 d \|{\bf s} -{\bf t}\|n+
\left(\mathbb{E} R_{4 d \|{\bf s} -{\bf t}\|n,1}^{\Pi,*}\right)^{\gamma/2}\right)\\
\le &
C(\theta,d)\|{\bf s} -{\bf t}\|^{2(d+1)},
\end{align*}
where $C(\theta,d)$ depends on its arguments only. 
Above we’ve used Lemmas \ref{L_F_1} and \ref{L_F_2},  Corollary  \ref{C_F_2} and \cite[Lemma 1({i})]{Chebunin2016}. Therefore, the first assertion follows.

We now approximate the original process by its Poissonization. Relative compactness of \(\{Y_{n,k}^{\Pi,*}\}_{n \geq n_0, k \geq 1}\) implies that for every \(\varepsilon, \eta > 0\), there exist \(\delta \in (0,1)\) and \(n_0\) such that for all \(n \geq n_0\),
\[
\mathbf{P}\bigl( \omega_{Y_{n,k}^{\Pi,*}}(\delta) \geq \eta \bigr) \leq \varepsilon.
\]
On the event \(\{\Pi(n\mathbf{s}) = [n\mathbf{t}]\}\), we have \(Y_{n,k}^{*}(\mathbf{t}) = Y_{n,k}^{\Pi,*}(\mathbf{s})\). Thus, by Lemma~\ref{L_F_4},
\begin{align*}
\mathbf{P}\Bigl( \sup_{\mathbf{t} \in \mathcal{I}_d} \bigl| Y_{n,k}^{*}(\mathbf{t}) - Y_{n,k}^{\Pi,*}(\mathbf{t}) \bigr| \geq \eta \Bigr)
&\leq \mathbf{P}\Bigl( \sup_{\mathbf{t} \in \mathcal{I}_d} \bigl| Y_{n,k}^{*}(\mathbf{t}) - Y_{n,k}^{\Pi,*}(\mathbf{t}) \bigr| \geq \eta,\, A(n) \Bigr) + \mathbf{P}(A(n)^c) \\
&\leq \mathbf{P}\bigl( \omega_{Y_{n,k}^{\Pi,*}}(\delta) \geq \eta \bigr) + \varepsilon \leq 2\varepsilon.
\end{align*}
This bound controls the approximation error. 
 
\end{proof}

We now restrict the random bounded outer measure \(R_{n\cdot,1}^*\) to \(\mathcal{B}^{\mathrm{fin}}[0,1]\), the subalgebra generated by finite unions of intervals (including possibly isolated points). Thus, every \(A \in \mathcal{B}^{\mathrm{fin}}[0,1]\) is a finite union of open intervals and points. The next lemma follows directly from the definitions of \(\mathcal{B}^{\mathrm{fin}}[0,1]\) and \(A_{\mathbf{t}}\) (with \(A_{\mathbf{t}} = \emptyset\) if \(d=0\)).

\begin{lemma}\label{R_fin}
If \(A \in \mathcal{B}^{\mathrm{fin}}[0,1]\), then there exists \(\mathbf{t} \in \mathcal{I}_d\) for some \(d \geq 0\) such that \(A \Delta A_{\mathbf{t}}\)  contains only a finite number of points.
\end{lemma}

\begin{corollary}[Multivariate CLT]\label{CLT1}
Let \(k_1, \dots, k_m \geq 1\) and \(A_1, \dots, A_m \in \mathcal{B}^{\mathrm{fin}}[0,1]\) for some \(m \geq 1\), and assume \eqref{reg} holds for some \(\theta \in (0,1)\). Then, as \(n \to \infty\),
\[
\bigl( Y_{n,A_1,k_1}^*, \dots, Y_{n,A_m,k_m}^* \bigr) \Rightarrow \bigl( Y_{A_1,k_1}^*, \dots, Y_{A_m,k_m}^* \bigr),
\]
with the same Gaussian limit as in Theorem~\ref{CLT}.
\end{corollary}

\begin{proof}
By Lemma~\ref{R_fin}, it suffices to consider the case \(A_j = A_{\mathbf{t}_j}\) for some \(\mathbf{t}_j \in \mathcal{I}_{d_j}\) (\(j=1,\dots,m\)). Let \(d := \max_{j \leq m} d_j\) and apply Theorem~\ref{FCLT}.
\end{proof}

Let $\mathbf{a} = \{a_i\}_{i\ge0}$ be a sequence of real numbers with \(a_0 = 0\). For \(\mathbf{t} \in \mathcal{I}_d\) (\(d \geq 1\)) and \(n \geq 1\), define the weighted sum statistic
\[
Q(n, \mathbf{a}, \mathbf{t}) := \frac{1}{\sqrt{M(n)}} \sum_{i=1}^n a_i \bigl( R_{n A_{\mathbf{t}}, i} - \mathbb{E} R_{n A_{\mathbf{t}}, i} \bigr).
\]

The following corollary generalizes \cite[Theorem 1]{Wang} to random fields over \(\mathcal{I}_d\) and broader probability distributions (beyond those in \cite{Wang}). It uses the integral representation of the limit from \cite{Wang}, where \(\pi_{ij}\) is defined in Corollary~\ref{KAA}.

\begin{corollary}\label{Q}
Assume \eqref{reg} holds for some \(\theta \in (0,1)\), \(d \geq 1\), and the sequence \(\mathbf{a}\) is such that
\begin{equation}\label{a_cov}
\lim_{m \to \infty} \sum_{i,j = m+1}^\infty \pi_{ij} a_i a_j = 0.
\end{equation}
Then, as \(n \to \infty\),
\[
\bigl( Q(n, \mathbf{a}, \mathbf{t}) : \mathbf{t} \in \mathcal{I}_d \bigr) \Rightarrow \bigl( Q(\mathbf{a}, \mathbf{t}) : \mathbf{t} \in \mathcal{I}_d \bigr),
\]
where
\begin{equation}\label{Q_aA}
Q(\mathbf{a}, \mathbf{t}) = \frac{1}{\sqrt{-\Gamma(-\theta)}} \int_{\mathbb{R}_+^d \times \Omega'} a^0_{\Pi(y A_{\mathbf{t}})} \, y^{-(\theta+1)/2} \, W(dy, d\omega)
\end{equation}
is a \(d\)-dimensional centered Gaussian field. Here, \(W\) is a standard Wiener measure on \(\mathbb{R}_+^d \times \Omega'\), \(\Pi\) is a standard Poisson process on \((\Omega', \mathcal{B}')\) independent of \(W\), and \(a^0_{\Pi(tA)} := a_{\Pi(tA)} - \mathbb{E} a_{\Pi(tA)}\). 
\end{corollary}

\begin{proof}
Note that
\[
Q(n, \mathbf{a}, A) = \sum_{i=1}^n a_i Y_{n,A,i}.
\]
By Theorem~\ref{FCLT}, finite-dimensional convergence holds, and it suffices to verify the tail condition
\begin{equation}\label{a_cond}
\lim_{m \to \infty} \lim_{n \to \infty} \mathrm{Var}\left( \frac{1}{\sqrt{M(n)}} \sum_{i=m+1}^n a_i \bigl( R_{nA,i}^{\Pi,*} - \mathbb{E} R_{nA,i}^{\Pi,*} \bigr) \right) = 0.
\end{equation}
By Corollary~\ref{KAA}, \eqref{a_cond} is equivalent (for \(|A| > 0\)) to \eqref{a_cov}. Thus,
\[
Q(n, \mathbf{a}, A) \Rightarrow \sum_{i=1}^\infty a_i Y_{A,i},
\]
a centered Gaussian field with covariances
\[
\mathrm{cov}\left( \sum_{i=1}^\infty a_i Y_{A,i}, \sum_{j=1}^\infty a_j Y_{B,j} \right) = \sum_{i,j=1}^\infty a_i a_j K(i,A,j,B).
\]

This matches \eqref{Q_aA}, since \(\mathbf{cov}(\int f \, dW, \int g \, dW) = \int f g\),
\begin{align*}
    &\mathbf{cov}\bigl( Q(\mathbf{a},A), Q(\mathbf{a},B) \bigr) \\
    &= \sum_{i,j=1}^\infty a_i a_j \frac{\theta}{\Gamma(1-\theta)} \int_0^\infty \bigl[ \mathbf{P}(\Pi(tA)=i, \Pi(tB)=j) - \mathbf{P}(\Pi(tA)=i) \mathbf{P}(\Pi(tB)=j) \bigr] t^{-\theta-1} \, dt,
\end{align*}
which equals \(\sum_{i,j} a_i a_j K(i,A,j,B)\) by definition of \(K\). 
\end{proof}

\begin{remark}
Note that \(q_k = -\binom{k-\theta-1}{k}\), since
 \(\binom{n}{k} = (-1)^k \binom{k-n-1}{k}\). From the asymptotics \(\binom{z+k}{k} \sim k^z / \Gamma(z+1)\) as \(k \to \infty\), we obtain
\begin{equation}\label{q_as}
q_k \sim -\frac{k^{-\theta-1}}{\Gamma(-\theta)} = \frac{\theta k^{-\theta-1}}{\Gamma(1-\theta)} \quad \text{as} \quad k \to \infty.
\end{equation}
Thus, \eqref{a_cov} holds due to \eqref{q_as} under the sufficient condition
\begin{equation}\label{a_suff}
\sum_{k=1}^\infty k^{-\theta-1} a_k^2 < \infty.
\end{equation}
\end{remark}

\section{Strong Law of Large Numbers}\label{sec5}

In this section, we prove SLLN. Theorem~\ref{T_SLLN} extends Karlin's SLLN~\cite{Karlin} to arbitrary sets \(A \in \mathcal{B}^{\mathrm{fin}}[0,1]\). We begin with a simple observation on Poissonized statistics.

\begin{lemma}\label{L_SLLN}
Let \(k \geq 1\), \(A \in \mathcal{B}[0,1]\) with \(|A| > 0\), and assume \eqref{reg} holds with \(\theta > 0\). For any sequence \(\{t_n\}_{n \geq 1}\) of positive numbers with \(\liminf_{n \to \infty} (\ln t_n)/(\ln n) > 0\), we have almost surely as \(n \to \infty\),
\begin{align*}
    \frac{R_{t_n A,k}^{\Pi,*}}{\mathbb{E} R_{t_n A,k}^{\Pi,*}} \to 1, \qquad \frac{R_{t_n A,k}^{\Pi}}{\mathbb{E} R_{t_n A,k}^{\Pi}} \to 1.
\end{align*}
\end{lemma}

\begin{proof}
Both \(R_{t_n A,k}^{\Pi,*}\) and \(R_{t_n A,k}^{\Pi}\) are sums of independent Bernoulli random variables. By the Rosenthal inequality and the fact that central absolute moments of Bernoullis are at most the mean, for any \(\gamma > 2\), we have
\begin{align*}
\mathbb{E} \left| \frac{R_{t_n A,k}^{\Pi,*}}{\mathbb{E} R_{t_n A,k}^{\Pi,*}} - 1 \right|^\gamma
&= \bigl( \mathbb{E} R_{t_n A,k}^{\Pi,*} \bigr)^{-\gamma} \mathbb{E} \Biggl| \sum_{i=1}^\infty \bigl( \mathbf{1}(\Pi_i(t_n A) \geq k) - \mathbb{P}(\Pi_i(t_n A) \geq k) \bigr) \Biggr|^\gamma \\
&\lesssim \bigl( \mathbb{E} R_{t_n A,k}^{\Pi,*} \bigr)^{-\gamma} \Biggl[ \sum_{i=1}^\infty \mathbb{P}(\Pi_i(t_n A) \geq k) + \Biggl( \sum_{i=1}^\infty \mathbb{P}(\Pi_i(t_n A) \geq k) \Biggr)^{\gamma/2} \Biggr] \\
&\sim \bigl( \mathbb{E} R_{t_n A,k}^{\Pi,*} \bigr)^{-\gamma/2}.
\end{align*}
By Lemma~\ref{L3}(B3) and the growth condition on \(t_n\) (ensuring \(\mathbb{E} R_{t_n A,k}^{\Pi,*} \to \infty\) with polynomial speed), the right-hand side tends to \(0\) for large \(\gamma\). Borel--Cantelli lemma then yields the first convergence a.s., the second follows analogously. 
\end{proof}

\begin{lemma}\label{L_SLLN_2}
Let \(k \geq 1\), \(A \in \mathcal{B}^{\mathrm{fin}}[0,1]\) with \(|A| > 0\), and assume \eqref{reg} holds with \(\theta > 0\). Then almost surely as \(t \to \infty\),
\[
\frac{R_{t A,k}^{\Pi,*}}{\mathbb{E} R_{t A,k}^{\Pi,*}} \to 1, \qquad \frac{R_{t A,k}^{\Pi}}{\mathbb{E} R_{t A,k}^{\Pi}} \to 1.
\]
\end{lemma}

\begin{proof}
By Lemma~\ref{R_fin}, there exists \(\mathbf{t} \in \mathcal{I}_d\) (\(d \geq 1\)) such that \(A \Delta A_{\mathbf{t}}\) contains only a finite number of points. Thus, it suffices to prove the result for \(A = A_{\mathbf{t}}\). Let \(t = t_n \in [n, n+1)\) for \(n \geq 1\). By Corollary~\ref{C_F_2} and Lemma~\ref{L_F_2},
\[
\bigl| R_{t_n A_{\mathbf{t}},k}^{\Pi,*} - R_{n A_{\mathbf{t}},k}^{\Pi,*} \bigr| \leq_{a.s.} R_{t_n A_{\mathbf{t}} \Delta n A_{\mathbf{t}}, 1}^{\Pi,*} \leq_{a.s.} R_{A_{\mathbf{t},n},1}^{\Pi,*},
\]
where \(A_{\mathbf{t},n} := \bigcup_{i=1}^d \bigl( [n\underline{t}_i, (n+1)\underline{t}_i] \cup [n\overline{t}_i, (n+1)\overline{t}_i] \bigr)\) satisfies \(|A_{\mathbf{t},n}| \leq 2d\). Moreover, \(\mathbb{E} R_{n A_{\mathbf{t}},k}^{\Pi,*} \sim \mathbb{E} R_{t_n A_{\mathbf{t}},k}^{\Pi,*}\) as \(n \to \infty\) by Lemma~\ref{L3}(B2) and \eqref{conv}, so \(R_{A_{\mathbf{t},n},1}^{\Pi,*} / \mathbb{E} R_{n A_{\mathbf{t}},k}^{\Pi,*} \to_{a.s.} 0\) by an argument analogous to Lemma~\ref{L_SLLN}.

The identity
\[
\frac{R_{t A,k}^{\Pi,*}}{\mathbb{E} R_{t A,k}^{\Pi,*}} = \frac{\mathbb{E} R_{[t] A,k}^{\Pi,*}}{\mathbb{E} R_{t A,k}^{\Pi,*}} \left( \frac{R_{[t] A,k}^{\Pi,*}}{\mathbb{E} R_{[t] A,k}^{\Pi,*}} + \frac{R_{t A,k}^{\Pi,*} - R_{[t] A,k}^{\Pi,*}}{\mathbb{E} R_{[t] A,k}^{\Pi,*}} \right)
\]
then yields the first claim via Lemma~\ref{L_SLLN}. The second follows similarly.
\end{proof}

\begin{theorem}\label{T_SLLN}
Let \(k \geq 1\), \(A \in \mathcal{B}^{\mathrm{fin}}[0,1]\) with \(|A| > 0\), and assume \eqref{reg} holds with \(\theta > 0\). Then almost surely as \(n \to \infty\),
\[
\frac{R_{nA,k}^*}{\mathbb{E} R_{nA,k}^*} \to 1, \qquad \frac{R_{nA,k}}{\mathbb{E} R_{nA,k}} \to 1.
\]
\end{theorem}

\begin{proof}
By Lemma~\ref{R_fin}, there exists \(\mathbf{t} \in \mathcal{I}_d\) (\(d \geq 1\)) such that \(A \Delta A_{\mathbf{t}}\) contains only a finite number of points. Therefore, it suffices to prove the result for \(A = A_{\mathbf{t}}\).

By the law of the iterated logarithm for \(\Pi\) and monotonicity, for any \(\varepsilon \in (0,1)\), there exists \(n_0\) such that
\[
\mathbf{P}\Bigl( \forall n \geq n_0 \, \exists \delta_n: |\delta_n| \leq 1,\, \Pi\bigl(n + 2\delta_n \sqrt{n \ln \ln n}\bigr) = n \Bigr) \geq 1 - \varepsilon.
\]
Analogously to Lemma~\ref{L_F_4}, there exists \(n_0 \geq 1\) such that
\[
\mathbf{P}\Bigl( \forall n \geq n_0 \, \exists \mathbf{s}: \|\mathbf{s} - \mathbf{t}\| \leq 2 \sqrt{(\ln \ln n)/n},\, \Pi(n\mathbf{s}) = [n\mathbf{t}] \Bigr) =: \mathbf{P}(B(n_0)) \geq 1 - \varepsilon.
\]
Define
\[
A_{\mathbf{t},n} := \bigcup_{i=1}^d \Bigl( [\underline{t}_i - 2\sqrt{(\ln \ln n)/n}, \underline{t}_i + 2\sqrt{(\ln \ln n)/n}] \cup [\overline{t}_i - 2\sqrt{(\ln \ln n)/n}, \overline{t}_i + 2\sqrt{(\ln \ln n)/n}] \Bigr).
\]
On the event \(\{\Pi(n\mathbf{s}) = [n\mathbf{t}]\}\), we have \(R_{n A_{\mathbf{t}},k}^* = R_{n A_{\mathbf{s}},k}^{\Pi,*}\). Thus, on \(B(n_0)\), Corollary~\ref{C_F_1} and Lemma~\ref{L_F_2} yield for \(n \geq n_0\),
\[
|R_{n A_{\mathbf{t}},k}^* - R_{n A_{\mathbf{t}},k}^{\Pi,*}| \le \sup_{{\bf s} : \|{\bf s}-{\bf t}\|\le 2 \sqrt{\ln\ln n / n}}|R^{\Pi,*}_{n A_{\bf s},k}-R^{\Pi,*}_{n A_{\bf t},k}|\le R^{\Pi,*}_{n A_{{\bf t},n},1}.
\]
By Lemma~\ref{L_SLLN_2}, for any \(\varepsilon, \eta > 0\), there exists \(n_0\) such that for \(n \geq n_0\),
\[
\mathbb{P}\Bigl( \sup_{n \geq n_0} \frac{|R_{n A_{\mathbf{t}},k}^* - R_{n A_{\mathbf{t}},k}^{\Pi,*}|}{\mathbb{E} R_{n A_{\mathbf{t}},k}^*} \geq \eta \Bigr) \leq \mathbb{P}\Bigl( \sup_{n \geq n_0} \frac{R_{n A_{\mathbf{t},n},1}^{\Pi,*}}{\mathbb{E} R_{n A_{\mathbf{t}},k}^*} \geq \eta,\, B(n_0) \Bigr) + \varepsilon \leq 2\varepsilon.
\]
This yields the first claim; the second follows analogously.
\end{proof}

\section{Applications}\label{sec6}

An important application of Karlin's infinite urn scheme arises in elementary probabilistic models of literary texts. These assume an infinite dictionary (urns), with words drawn independently according to a regularly varying distribution with parameter \(\theta \in (0,1)\); successive words in the text are balls. The number of distinct words corresponds to the number of non-empty urns. Abebe et al.~\cite{Abebe} showed that this model fits the real texts well. 

The process of interest is \(R_n(t) := R_{n[0,t],1}^*\) for \(t \in [0,1]\), which corresponds to Theorem~\ref{FCLT} with \(d=1\). The standardized and centered version converges weakly to a Gaussian process \(Z(t)\) with covariance from Corollary~\ref{K*}
\[
K(s,t) := \mathbb{E} Z(s) Z(t) = (s+t)^\theta-(s \vee t)^\theta, \quad s,t \in [0,1].
\] 

For concatenated texts (e.g., from different authors or LLM prompts with varying instructions), the process deviates from the model. The unknown mean \(\mathbb{E} R_n(t)\) hinders diagnostics, but Abebe et al.~\cite{Abebe} proposed the backward process \(R_n'(t) := R_{n[1-t,1],1}^*\), counting distinct words from the end.

Under \eqref{reg}, the self-normalized centered vector process
\[
\left\{ \left( \frac{R_n(t) - \mathbb{E} R_n(t)}{\sqrt{M(n)}}, \frac{R_n'(t) - \mathbb{E} R_n'(t)}{\sqrt{M(n)}} \right) : t \in [0,1] \right\}
\]
converges weakly in \(D[0,1]^2\) in uniform metric to a centered 2-dimensional Gaussian process \((Z, Z')\) with \(\mathbb{E} Z'(s) Z'(t) = K(s,t)\) and cross-covariance
\[
K'(s,t) := \mathbb{E} Z(s) Z'(t) = \bigl( (s+t)^\theta - 1 \bigr) \mathbf{1}(s+t > 1).
\]  This recovers \cite[Theorem 1]{Abebe} as a corollary of Theorem~\ref{FCLT} (\(d=2\), four endpoints). Abebe et al.~\cite{Abebe} applied continuous functionals of \((R_n(t), R_n'(t))\) to test text homogeneity. 

A limitation of the above method arises when solving the problem of embedded foreign text detection, where the counting must start from an unknown embedding point. We propose a circular text statistic: for each starting position 
\(\lfloor ns \rfloor\) (\(s \in [0,1)\)), we count the number of distinct words forward to \(\lfloor n(s+t) \rfloor \mod n\) (\(t > 0\)) minus the number of words backward to \(\lfloor n(s-t) \rfloor \mod n\).  For \(s,t \in [0,1]\), define the circular arcs
\begin{align*}
A(s,t) &:= 
\begin{cases} 
[s, s+t], & s+t \leq 1; \\ 
[s,1] \cup [0, s+t-1], & s+t > 1;
\end{cases} \\
B(s,t) &:= 
\begin{cases} 
[s-t, s], & s-t \geq 0; \\ 
[0,s] \cup [s-t+1, 1], & s-t < 0.
\end{cases}
\end{align*}
The difference statistic is \(U_n(s,t) := R_{n A(s,t),1}^* - R_{n B(s,t),1}^*\).

By Theorem~\ref{FCLT} (\(d=2\), \(m=2\), \(k_1 = k_2 = 1\)), the self-normalized field
\(
\{
{U_n(s,t)}/{\sqrt{R_n}}, \quad s,t \in [0,1]
\}
\)
converges weakly in the uniform metric on \([0,1]^2\) to a centered Gaussian random field with \(\theta\)-dependent covariance. Continuous functionals thereof, such as
\[
\frac{1}{R_n} \int_0^1 \int_0^1 U_n(s,t)^2 \, ds \, dt,
\]
provide tests for text homogeneity under the model. 

\section*{Conclusion}

Discrete outer random measures from the infinite urn scheme substantially generalize classical finite-dimensional occupancy statistics like the number of non-empty urns. Their restrictions to Borel measurable or \(\mathcal{B}^{\mathrm{fin}}[0,1]\) sets yield random fields amenable to limit theorems. In particular, the functional central limit theorem (Theorem~\ref{FCLT}) enables precise goodness-of-fit tests for data against the infinite urn model, with applications to text analysis and beyond.

\section*{Declarations}

\begin{itemize}
\item The work of the third author is supported by the Mathematical Center in Akademgorodok under the agreement No. 075-15-2025-348 with the Ministry of Science and Higher Education of the Russian Federation.
\item Conflict of interest/Competing interests (check journal-specific guidelines for which heading to use): no.
\item Ethics approval and consent to participate: yes.
\item Consent for publication: yes.
\item Data availability: not applicable. 
\item Materials availability: not applicable.
\item Code availability: not applicable. 
\item Author contribution: Artyom Kovalevskii is the author of the general idea of the study. Berhane Abebe is the author of the proofs of Lemmas \ref{L1}, \ref{L2}, \ref{L3}, 
\ref{R_fin}, \ref{L_SLLN},  
Proposition~\ref{corr}, Theorem \ref{CLT}, Corollaries 
\ref{K*}, \ref{KAA}, \ref{CLT1}, \ref{Q}. Mikhail Chebunin is the author of the proofs of Lemmas \ref{L_F_1}, \ref{L_F_2}, \ref{L_F_3}, \ref{L_F_4}, \ref{L_SLLN_2}, Theorems \ref{FCLT} and  \ref{T_SLLN}, Corollaries 
\ref{C_F_1}, \ref{C_F_2}.
\end{itemize}

\bibliography{sn-bibliography}

\end{document}